\documentclass[11pt]{article}
\usepackage[T1]{fontenc}
\usepackage[utf8]{inputenc}
\usepackage{lmodern}
\usepackage[margin=1in]{geometry}
\usepackage{amsmath,amssymb,amsthm,mathtools}
\usepackage{microtype}
\usepackage{booktabs,array,tabularx}
\usepackage{float}
\usepackage{enumitem}
\usepackage{aliascnt,etoolbox,needspace}
\usepackage[dvipsnames]{xcolor}
\usepackage[numbers,sort&compress]{natbib}
\usepackage{xurl}
\definecolor{LinkBlue}{RGB}{28,66,132}
\definecolor{CiteTeal}{RGB}{0,94,92}
\usepackage[colorlinks=true,linkcolor=LinkBlue,citecolor=CiteTeal,urlcolor=LinkBlue,
  bookmarksnumbered=true,pdfencoding=auto,psdextra]{hyperref}
\usepackage[nameinlink,capitalise,noabbrev]{cleveref}
\hypersetup{pdftitle={A 27 x 27 x 27 counterexample to Comon's Conjecture},pdfauthor={Benjamin Lovitz}}
\providecommand{\doi}[1]{\href{https://doi.org/#1}{\nolinkurl{doi:#1}}}
\providecommand{\arxiv}[1]{\href{https://arxiv.org/abs/#1}{\nolinkurl{arXiv:#1}}}
\setlist[enumerate]{label=\textup{(\roman*)},leftmargin=2.1em,itemsep=2pt,topsep=4pt}
\numberwithin{equation}{section}
\newtheorem{theorem}{Theorem}[section]
\newaliascnt{lemma}{theorem}
\newtheorem{lemma}[lemma]{Lemma}\aliascntresetthe{lemma}
\newaliascnt{proposition}{theorem}
\aliascntresetthe{proposition}
\newaliascnt{corollary}{theorem}
\aliascntresetthe{corollary}
\theoremstyle{definition}
\newaliascnt{definition}{theorem}
\newtheorem{definition}[definition]{Definition}\aliascntresetthe{definition}
\theoremstyle{remark}
\newaliascnt{remark}{theorem}
\newtheorem{remark}[remark]{Remark}\aliascntresetthe{remark}
\crefname{theorem}{Theorem}{Theorems}\Crefname{theorem}{Theorem}{Theorems}
\crefname{lemma}{Lemma}{Lemmas}\Crefname{lemma}{Lemma}{Lemmas}
\crefname{proposition}{Proposition}{Propositions}\Crefname{proposition}{Proposition}{Propositions}
\crefname{corollary}{Corollary}{Corollaries}\Crefname{corollary}{Corollary}{Corollaries}
\crefname{definition}{Definition}{Definitions}\Crefname{definition}{Definition}{Definitions}
\crefname{remark}{Remark}{Remarks}\Crefname{remark}{Remark}{Remarks}
\crefname{equation}{Equation}{Equations}\Crefname{equation}{Equation}{Equations}
\crefname{section}{Section}{Sections}\Crefname{section}{Section}{Sections}
\crefname{table}{Table}{Tables}\Crefname{table}{Table}{Tables}
\BeforeBeginEnvironment{theorem}{\Needspace{7\baselineskip}}
\BeforeBeginEnvironment{proposition}{\Needspace{6\baselineskip}}
\BeforeBeginEnvironment{lemma}{\Needspace{5\baselineskip}}
\BeforeBeginEnvironment{corollary}{\Needspace{5\baselineskip}}
\BeforeBeginEnvironment{definition}{\Needspace{5\baselineskip}}
\BeforeBeginEnvironment{remark}{\Needspace{4\baselineskip}}
\newcommand{\CC}{\mathbb C}
\newcommand{\RR}{\mathbb R}
\newcommand{\QQ}{\mathbb Q}
\newcommand{\kk}{\mathbb F}
\newcommand{\Sym}{S}
\newcommand{\sym}{\Pi}
\newcommand{\rk}{\mathrm{R}}
\newcommand{\srk}{\mathrm{R}_{\mathrm{s}}}
\newcommand{\srkfield}[1]{\mathrm{R}_{\mathrm{s},#1}}
\newcommand{\spn}{\operatorname{span}}

\newcommand{\id}{\operatorname{id}}
\newcommand{\cA}{\mathcal A}

\theoremstyle{definition}
\newaliascnt{claim}{theorem}
\newtheorem{claim}[claim]{Claim}
\aliascntresetthe{claim}
\newaliascnt{conjecture}{theorem}
\newtheorem{conjecture}[conjecture]{Conjecture}
\aliascntresetthe{conjecture}

\title{A $27 \times 27 \times 27$ counterexample to Comon's conjecture}
\author{Benjamin Lovitz\thanks{Concordia University. benjamin.lovitz@concordia.ca}}
\date{\today}

\begin{document}
\maketitle
\begin{abstract}
We report an explicit construction of a $27 \times 27 \times 27$ symmetric tensor with rational entries that has tensor rank $55$ over $\QQ$ and symmetric tensor rank $56$ over $\CC$, providing a small counterexample to Comon's conjecture over $\QQ$, $\RR$, and $\CC$. The construction follows the framework of symmetric adjoins introduced by Shitov. The primary technical contribution of this work is to prove a special case of Shitov's Conjecture 6 in~\cite{Shitov2018}.
\end{abstract}

\textbf{Use of artificial intelligence:} The author used GPT-6 Astra interactively throughout every stage of preparing this manuscript, including to obtain an initial proof of the main result, which the author subsequently verified and simplified. The author takes full responsibility for the correctness, exposition, and attribution in this  manuscript.

\section{Introduction}\label{sec:intro}

Let $\kk \in \{\QQ, \RR, \CC\}$. Given an order-$d$ tensor $T \in (\kk^n)^{\otimes d}$, we say that $T$ is \textit{symmetric} if $T_{i_1,\dots, i_d} = T_{i_{\sigma(1)},\dots, i_{\sigma(d)}}$ for all $i_1,\dots, i_d \in \{1,\dots, n\}$ and permutations $\sigma \in \mathfrak{S}_d$, i.e. the coordinates do not depend on the order of the indices. We say that $T$ has \textit{rank one} if it can be written as $T=v_1\otimes \dots \otimes v_d$ for some nonzero vectors $v_i \in \kk^n$. We say that $T$ has \textit{symmetric rank one} if it can be written as $T=\lambda v^{\otimes d}$ for some nonzero $\lambda \in \kk$ and nonzero $v \in \kk^n$. The \textit{rank} $\rk_{\kk}(T)$ is the smallest $N$ such that $T$ can be written as a sum of $N$ rank one tensors. The \textit{symmetric rank} $\srkfield{\kk}(T)$ is the smallest $N$ such that $T$ can be written as a sum of $N$ symmetric rank one tensors.

It is clear that $\rk_{\kk}(T) \leq \srkfield{\kk}(T)$ for any symmetric tensor $T$. For symmetric matrices (i.e., $d=2$), the rank and symmetric rank are equal. The question of whether equality holds for higher order ($d>2$) symmetric tensors was first asked by Comon, and the assertion that equality holds has come to be known as Comon's conjecture.

Early results establishing equality of rank and symmetric rank in special cases were obtained by Comon et al.~\cite{CGLM2008}, and later in \cite{BB2013,Friedland2016,ZHQ2016,Seigal2019,Seigal2020}, as well as over other fields~\cite{ZHSX2020}. Related questions about partially symmetric ranks were considered in~\cite{BL2013,GOV2019,BC2022}, and border rank analogues were studied in~\cite{BGL2013,Seigal2019,Seigal2020,MandziukVentura2026}. See also~\cite{CMM2018,ShitovMore2024} for more in-depth accounts of prior work.

One (and perhaps, the only) known method of proving counterexamples to Comon's conjecture was introduced by Shitov in~\cite{Shitov2018}, although the specific example presented in that paper turned out to have an error~\cite{Draisma2024,ShitovIllusions2026}. Using this technique, Shitov found an order 4 counterexample over $\RR$~\cite{Shitov2020}. In subsequent works, Wang and Seigal established a real counterexample of order 6~\cite{WS2023}, and Shitov found counterexamples of order 3 in sufficiently large dimension over any infinite field of characteristic not equal to 2 or 3~\cite{ShitovPartII2024}. More recently, Shitov reduced the size of his real order 4 counterexample~\cite{ShitovMonomials2026}, established the existence of real counterexamples in every even order $d \geq 4$~\cite{ShitovMonomials2026}, and established an explicit complex counterexample of order 4~\cite{ShitovFurther2026}. We summarize these results in~\Cref{tab:comparison}.

\begin{table}[tb]
\centering\small
\setlength{\tabcolsep}{5pt}
\begin{tabular}{@{}lccrrrl@{}}
\toprule
Construction & Field & Order & Variables & $\rk$ & $\srk$ \\
\midrule
Shitov \cite{Shitov2020} & $\RR$ & 4 & 208 & 761 & 762 \\
Wang--Seigal \cite{WS2023} & $\RR$ & 6 & 5178 & 30913 & 30914 \\
Shitov \cite{ShitovPartII2024} & $\kk$ & 3 & *
& \multicolumn{2}{c}{$\rk<\srk$} \\
Shitov \cite{ShitovMonomials2026} & $\RR$ & 4 & 190 & 721 & 722 \\
Shitov \cite{ShitovMonomials2026} & $\RR$ &even $d\geq4$ & * & \multicolumn{2}{c}{$\rk<\srk$} \\
Shitov \cite{ShitovFurther2026} & $\CC$ & 4 & 2430 & 9603 & $\ge9604$ \\
\Cref{thm:main} & $\CC$ & 3 & 27 & 55 & 56 \\
\bottomrule
\end{tabular}
\caption{Existing counterexamples to Comon's conjecture. In the third row, $\kk$ is any infinite field of characteristic not equal to 2 or 3. Asterisks denote unspecified integers.}
\label{tab:comparison}
\end{table}

A drawback to these counterexamples is that they require the dimension $n$ to be large ($>100$). Our main result addresses this point in the degree 3 case.
\begin{theorem}\label{thm:main}
There exists a symmetric tensor $T \in (\CC^{27})^{\otimes 3}$ with rational entries for which
\[
 \rk_{\kk}(T)=55<56=\srkfield{\kk}(T)
\]
for every field $\kk \in \{\QQ, \RR, \CC\}$.
\end{theorem}

To our knowledge, it remains open whether the rank and symmetric rank can differ by greater than one (see~\cite[Problem 23]{Shitov2018}).

\subsection{Methods and technical contribution}\label{sec:changes}

Our construction uses the symmetric adjoin framework introduced by Shitov in~\cite{Shitov2018}. This is a general method for constructing counterexamples to Comon's conjecture, which proceeds as follows. We focus on the order 3 case, although Shitov also introduced a higher-order analogue which has been useful for producing counterexamples.

In Shitov's symmetric adjoin framework, one starts with a symmetric tensor $T_0 \in \CC^k \otimes \CC^k \otimes \CC^k$ and a linearly independent set of symmetric rank-one matrices $M_1,\dots, M_m \in \CC^k \otimes \CC^k$. One then constructs the symmetric tensor
\[
T=T_0+3\sum_{\nu=1}^m e_{\nu} M_{\nu} \in (\CC^k \oplus \CC^m)^{\otimes 3},
\]
where $e_1,\dots, e_m$ is a basis for $\CC^m$ and $e_{\nu} M_{\nu}:=\Pi_3(e_{\nu} \otimes M_{\nu})$ denotes the symmetric product, where $\Pi_3$ is the projection onto the symmetric subspace. Shitov then considers a certain affine linear space
\[
\cA:=T_0 + \sum_{\nu=1}^m (\CC^k \otimes \CC M_{\nu} + \iota_2(\CC^k \otimes \CC M_{\nu}) + \CC M_{\nu} \otimes \CC^k) \subseteq \CC^k \otimes \CC^k \otimes \CC^k,
\]
where $\iota_2$ is the permutation defined as $\iota_2(a \otimes b\otimes c):=b\otimes a \otimes c$ and extended linearly.  In Shitov's notation, $\cA=\mathcal{E}_{\spn\{M_1,\dots, M_m\}}(T_0)$.

Shitov established that the tensor rank of $T$ is equal to $3m$ plus the minimum rank of a tensor in $\cA$~\cite[Lemma 5]{Shitov2018}, and conjectured that the symmetric tensor rank of $T$ equals $3m$ plus the minimum symmetric rank of a symmetric tensor in $\cA$~\cite[Conjecture 6]{Shitov2018}. A positive resolution of this conjecture would present a promising route towards a counterexample to Comon's conjecture. For example, certifying that the minimum rank of a tensor in $\cA$ is equal to 1, while the minimum symmetric rank of a tensor in $\cA$ is equal to 2, seems much easier than directly checking if a (potentially high-rank) tensor $T$ exhibits a gap between its rank and symmetric rank. Indeed, Shitov writes that a positive answer to Conjecture 6 ``would present a much simpler counterexample to Comon's conjecture than the one constructed here.'' Instead, he relies on more complicated arguments to establish this result in specific examples. An excellent review of Shitov's approach is given in~\cite[Section 4]{WS2023}.

In this work, we prove exactly the special case of Shitov's Conjecture 6 that was given as an example in the above paragraph, namely, when the minimum symmetric rank of a symmetric tensor in $\cA$ is equal to $2$. As a byproduct, we obtain a small counterexample to Comon's conjecture.

We remark that Shitov disproved a slightly more general statement than Conjecture 6 in~\cite{ShitovRefutation2024} (in addition to his related Conjecture 7), namely when the matrices $M_1,\dots, M_m$ are not required to have symmetric rank 1. We note that~\cite[Lemmas 11 and 15]{Shitov2018} combined with the very recent preprint~\cite[Section 9]{ShitovIllusions2026} shows that the construction in~\cite{Shitov2018} gives a counterexample to the original rank-one version of Conjecture 6.

\section*{Acknowledgments}
The author thanks Alexander Taveira Blomenhofer for helpful feedback on an initial draft of this manuscript. The author acknowledges support from NSERC Discovery Grant number
RGPIN-2026-05413.

\section{Preliminaries}\label{sec:background}

In this section we introduce some mathematical preliminaries and definitions required in this work.

\subsection{Tensors and the symmetric product}

For a finite-dimensional vector space $V$ over $\kk \in \{\QQ, \RR, \CC\}$, let $S^dV\subseteq V^{\otimes d}$ be the \textit{symmetric subspace} of tensors for which $\sigma\cdot T =T$ for all permutations $\sigma \in \mathfrak{S}_d$, where $\sigma$ acts on $T$ by permuting factors, i.e. $(\sigma \cdot T)_{i_1,\dots, i_d}=T_{i_{\sigma(1)},\dots, i_{\sigma(d)}}$. We let
\[
 \sym_d:V^{\otimes d}\longrightarrow S^dV,
 \qquad \sym_d(T)=\frac1{d!}\sum_{\sigma\in\mathfrak S_d}\sigma T
\]
be the projection onto the symmetric subspace. For vectors $v_1 \in S^{\ell_1}(V),\dots, v_p \in S^{\ell_p}(V)$ we let
\begin{align*}
v_1 v_2 \cdots v_p := \Pi_{\ell_1+\dots+\ell_p}(v_1 \otimes \dots \otimes v_p) \in S^{\ell_1+\dots+\ell_p}(V).
\end{align*}
For example, for $u,v,w \in V$ we have
\begin{align*}
 uv&:=\Pi_2(u\otimes v)=\frac12(u\otimes v+v\otimes u)\\
 uvw&:=\Pi_3(u \otimes v \otimes w)\\
 &= \frac16(u\otimes v\otimes w+u\otimes w\otimes v+v\otimes u\otimes w+ v \otimes w \otimes u + w \otimes u \otimes v + w \otimes v \otimes u).
\end{align*}

\newpage
\subsection{Rank and flattenings}

\begin{definition}[Tensor rank and symmetric rank]\label{def:ranks}
For $T\in V_1 \otimes \dots \otimes V_d$, its \textit{(tensor) rank} $\rk(T)$ is the least $r$
for which
\[
 T=\sum_{i=1}^r v_{i,1}\otimes\cdots\otimes v_{i,d}
\]
for some $v_{i,j} \in V_j$. For $T\in S^dV$, its \textit{symmetric rank} $\srk(T)$ is the least $r$ for which
\[
 T=\sum_{i=1}^r\lambda_i v_i^{\otimes d}
\]
for some $\lambda_i \in \kk$ and $v_i \in V$. When the ground field is not clear from the context, we
write $\rk_{\kk}(T)$ and $\srkfield{\kk}(T)$.
\end{definition}

For $T\in V_1 \otimes V_2 \otimes V_3$, write
\begin{equation*}
 T_3:V_3^*\longrightarrow V_1\otimes V_2,
 \qquad T_3(\beta)=(\id\otimes\id\otimes\beta)T
\end{equation*}
for the third flattening, and define the other flattenings $T_1, T_2$ analogously.

Define three permutation maps $\iota_i: V \otimes V\otimes V \rightarrow V \otimes V \otimes V$ by
\begin{align*}
 \iota_1(u \otimes a\otimes b)&=u\otimes a\otimes b,\notag\\
 \iota_2(u \otimes a\otimes b)&=a\otimes u\otimes b,\label{eq:insertions}\\
 \iota_3(u \otimes a\otimes b)&=a\otimes b\otimes u,\notag
\end{align*}
extended linearly. 

For subspaces $L \subseteq S^{d_1}(V), H \subseteq S^{d_2}(V)$, we define $LH:=\spn\{u v : u \in L, v \in H\}$. Note that $LH=HL$. For a vector $v \in S^{d_1}(V)$ we use the shorthand $vH:=\spn\{v\} H$.

\subsection{Shitov's symmetric adjoin construction}

In this section we review a special case of the symmetric adjoin construction defined in~\cite[Section 1.3]{Shitov2018}.

\begin{definition}\label{def:adjoining}
Let $V$ be a finite-dimensional vector space over $\kk \in \{\QQ, \RR,\CC\}$, let $T_0\in S^3V$, and let $K\subseteq S^2V$ be an $m$-dimensional subspace with a basis of the form
\[
 M_\nu= \lambda_{\nu} w_\nu^{2},\qquad
 1\le\nu\le m.
\]
Let
\begin{equation*}
 J(K):=\sum_{j=1}^3\iota_j(V \otimes K),\qquad
 \mathcal{A}(T_0,K):=T_0+J(K).
\end{equation*}
When the choices of $T_0$ and $K$ are clear from the context, we write $\cA:=\mathcal{A}(T_0,K)$. Let $Z:=\kk^m$ and let $\widetilde V:=V \oplus Z$. We define the \textit{symmetrically adjoined tensor} $T\in S^3(\widetilde V)$ associated to $V$, $T_0$, and the matrices $M_\nu$, as
\begin{align*}
 T&=T_0+\sum_{\nu=1}^m\sum_{j=1}^3\iota_j(e_\nu \otimes M_\nu)\\
 &=T_0+3\sum_{\nu=1}^m \lambda_{\nu} e_\nu w_{\nu}^2 \in S^3(\widetilde V),
\end{align*}
where $e_1,\dots, e_m$ is the standard basis of $Z$.
\end{definition}
Note that $\mathcal{A}(T_0,K)=\mathcal{E}_K(T_0)$ in the notation of~\cite[Section 1.2]{Shitov2018}.

We will use the following elementary fact many times:
\[
J(K)\cap S^3V=\Pi_3(J(K))=VK.
\]
Here, the first equality follows from the fact that $J(K)$ is invariant under permutations of the three tensor factors, and the second is straightforward.

\section{A special case of a conjecture of Shitov}\label{sec:projection}

In this section we fix a finite-dimensional vector space $V$ over $\kk\in \{\RR, \CC\}$, along with $T_0$, $K$, the $m$ symmetric rank one matrices $M_\nu$, and the resulting symmetrically adjoined tensor $T$, as described in \Cref{def:adjoining}, along with the other notation introduced there. Over the complex numbers, Shitov conjectured the following:

\begin{conjecture}[Conjecture 6 in \cite{Shitov2018}]\label{con:shitov-conjecture}
\begin{align*}
\srkfield{\CC}(T)=3m+\min_{P \in \cA \cap S^3(V)} \srkfield{\CC}(P).
\end{align*}
\end{conjecture}

Here we prove this conjecture in the special case when $\min_{P \in \cA \cap S^3(V)} \srkfield{\CC}(P)=2$ (when this quantity is zero or one, the conjecture holds by an analogous argument to~\cite[Proposition 3.19]{WS2023}). We note that~\cite[Lemmas 11 and 15]{Shitov2018} combined with the very recent preprint~\cite[Section 9]{ShitovIllusions2026} shows that the construction in~\cite{Shitov2018} gives a counterexample to Conjecture 6 in the case when $\min_{P \in \cA \cap S^3(V)} \srkfield{\CC}(P)\geq4$.

In the following, the statements about $\rk_{\kk}(T)$ are a direct consequence of~\cite[Lemma 5]{Shitov2018} (along with its real counterpart, which is proven identically). In addition, the upper bound $\srkfield{\kk}(T)\leq 3m+2$ in the case when $\cA$ contains a symmetric rank two tensor is also known (see e.g.~\cite[Lemma 6.12]{Shitov2020} and~\cite[Proposition 3.18]{WS2023}). We include these statements in the theorem for convenience.

\begin{theorem}\label{thm:threshold}
If  $\cA$ contains neither zero nor a symmetric rank one tensor, then the symmetrically adjoined tensor $T \in S^3(\widetilde V)$ satisfies
\begin{equation*}
 \rk_{\kk}(T)\ge3m+1,
 \qquad\text{and}\qquad \srkfield{\kk}(T)\ge3m+2.
\end{equation*}
If $\cA$ contains a (non-symmetric) rank one tensor, then equality holds in the first bound, and if $\cA$ contains a symmetric rank two tensor, then equality holds in the second bound.
\end{theorem}

\begin{remark}
We remark on the more general setting when the basis elements $M_{\nu}$ of $K$ are not required to have rank one. A nearly identical proof as below shows that $\srkfield{\kk}(T)\ge3m+2$ also holds in this case, establishing the lower bound
\[
\srkfield{\kk}(T)\geq 3m+\min_{P \in \cA \cap S^3(V)} \srkfield{\kk}(P)
\]
in the case $\min_{P \in \cA \cap S^3(V)} \srkfield{\kk}(P)=2$. Interestingly, Shitov recently found a counterexample to this lower bound over $\CC$ in the case $\min_{P \in \cA \cap S^3(V)} \srkfield{\CC}(P)=3$~\cite[Example 3]{ShitovRefutation2024}. Wang and Seigal conjectured an analogous lower bound over the real numbers~\cite[Conjecture 3.17]{WS2023}, for which our theorem also establishes a special case.
\end{remark}

To prove this theorem, we begin with the following lemma, which is essentially~\cite[Lemma 5]{Shitov2018}. The proof of this lemma is a straightforward application of the substitution method~\cite[Lemma 2]{HK1971} (see also the discussion preceding~\cite[Lemma 4]{Shitov2018}). We provide a self-contained proof in~\Cref{sec:substitution} for completeness.

\begin{lemma}\label{lem:projections}
Let
\begin{align*}
T= \sum_{i=1}^N a_i\otimes b_i\otimes c_i
\end{align*}
be an expression of the symmetrically adjoined tensor $T$ as a sum of $N$ product tensors. Then there exist linear maps
\[
 \pi_j:\widetilde V\longrightarrow V,
 \qquad \pi_j|_V=\id_V,\qquad j=1,2,3,
\]
and pairwise disjoint subsets
\[
I_1, I_2, I_3 \subseteq [N], \qquad |I_j|=m
\]
such that the following statements hold:
\begin{enumerate}
\item $\pi_1 a_i=0$ for all $i \in I_1$, $\quad$ $\pi_2 b_i =0$ for all $i \in I_2$, \quad and \quad $\pi_3 c_i = 0$ for all $i \in I_3$.
\item $(\pi_1 \otimes \pi_2 \otimes \pi_3)(T) \in \cA$.
\item It holds that
\begin{equation}\label{eq:ordinary-substitution}
 \rk(T)=3m+\min\{\rk(P):P\in\cA\}.
\end{equation}
\end{enumerate}
\end{lemma}

We note that in the case when the $M_{\nu}$ do not have rank one,~\Cref{eq:ordinary-substitution} becomes just an inequality $\geq$. The rest of the statement is identical.

With this lemma in hand, we now give a brief overview of the proof of~\Cref{thm:threshold}. As mentioned above, the main task is to rule out a decomposition $T=\sum_{i=1}^N h_i^3$ with $N\leq 3m+1$.
Assuming toward contradiction that such a decomposition exists, the substitution method leaves a single nonzero rank one tensor $P=u_1\otimes u_2\otimes u_3\in\cA$. Since $\cA$ does not contain a symmetric rank one tensor, we can assume that $u_1$ and $u_2$ are linearly independent. We show that $W:=\spn\{u_1,u_2,u_3\}$ is two-dimensional and that
$K\cap S^2W$ is either $\{0\}$ or $\kk x^2$ for some $0\ne x\in W$, with exactly one factor of $P$ proportional to $x$ in the latter case~(\Cref{lem:residual-plane}). From here, it follows that the nonzero rank one matrices $\pi_1(h_i) \otimes \pi_2(h_i)$ appearing after the first two projections in the substitution method form a basis of
$K\oplus\kk(u_1\otimes u_2)$~(\Cref{lem:partial-projection}). Comparing the third-factor coefficients in this basis gives a contradiction in each case: The first forces a square $h_i^2$ to be proportional to the non-symmetric matrix $u_1 \otimes u_2$, while the second forces a second factor of $P$ to be proportional to $x$.

\begin{lemma}\label{lem:partial-projection}
Suppose that $0\notin\cA$, and there exists a symmetric decomposition
\[
    T=\sum_{i=1}^N h_i^{3}
\]
with $h_i$ non-zero and $N \leq 3m+1$. The following statements hold:

\begin{enumerate}

\item  $N=3m+1$, and there exist nonzero $u_1,u_2,u_3 \in V$ such that
\[
    P:=(\pi_1\otimes\pi_2\otimes\pi_3)(T)=u_1 \otimes u_2 \otimes u_3 \in\cA,
\]
where the maps $\pi_i$ are as in~\Cref{lem:projections}.

\item Let
\[
    S:=\{i:\pi_1(h_i)\ne0\text{ and }\pi_2(h_i)\ne0\}.
\]
If $u_1 \otimes u_2 \notin K$, then $|S|=m+1$ and the set
\begin{align}\label{eq:rank-one-basis}
    \{\pi_1(h_i)\otimes\pi_2(h_i):i\in S\}
\end{align}
forms a basis for $K\oplus \kk (u_1 \otimes u_2)$.
\end{enumerate}
\end{lemma}

\begin{proof}
By the construction in \Cref{lem:projections}, there exist pairwise
disjoint sets
\[
    I_1,I_2,I_3\subseteq\{1,\ldots,N\},
    \qquad |I_j|=m,
\]
such that $\pi_j(h_i)=0$ for all $i \in I_j$.
Therefore
\[
    P=
    \sum_{i\notin I_1\cup I_2\cup I_3}
    \pi_1(h_i)\otimes\pi_2(h_i)\otimes\pi_3(h_i).
\]
There are at most $N-3m\le1$ terms in this sum.
Since $P\in\cA$ and $0\notin\cA$, the sum has exactly 1 term, so $N=3m+1$ and $P$ is a non-zero product tensor which we denote by $u_1 \otimes u_2 \otimes u_3$.

Let
\[
    G:=(\pi_1\otimes\pi_2\otimes\mathrm{id}_{\widetilde V})(T)= \sum_{i=1}^N
       \pi_1(h_i)\otimes\pi_2(h_i)\otimes h_i.
\]
A term is nonzero precisely when $i\in S$. Moreover, $S$ is disjoint from $I_1\cup I_2$,
so
\[
    |S|\le N-2m=m+1.
\]
Note that
\[
\begin{aligned}
    G=T_0+\sum_{\nu=1}^m\Bigl(
      &\iota_1(\pi_1(e_\nu) \otimes M_\nu)
       +\iota_2(\pi_2(e_\nu) \otimes M_\nu)
       +M_\nu\otimes e_\nu
    \Bigr),\\
    P=T_0+\sum_{\nu=1}^m\Bigl(
      &\iota_1(\pi_1(e_\nu) \otimes M_\nu)
       +\iota_2(\pi_2(e_\nu) \otimes M_\nu)
       +M_\nu\otimes\pi_3(e_\nu)
    \Bigr),
\end{aligned}
\]
and hence,
\begin{equation}\label{eq:partial-projection}
    G=P+
      \sum_{\nu=1}^m M_\nu\otimes(e_\nu-\pi_3(e_\nu)).
\end{equation}

Now assume $u_1 \otimes u_2 \notin K$. Then
\[
    \operatorname{im}G_3
    =\operatorname{span}\{u_1 \otimes u_2 ,M_1,\ldots,M_m\}
    =K\oplus\kk( u_1 \otimes u_2).
\]
To see this, note that the vectors
\begin{align*}
    u_3 ,\quad e_1-\pi_3(e_1),\ldots,e_m-\pi_3(e_m)
\end{align*}
are linearly independent. Indeed, a relation
\[
    \alpha u_3+\sum_{\nu=1}^m
       \beta_\nu(e_\nu-\pi_3(e_\nu))=0
\]
projects to
\[
    \sum_{\nu=1}^m\beta_\nu e_\nu=0
\]
in $Z$. Since $e_1,\ldots,e_m$ forms a basis of $Z$, all $\beta_\nu$
vanish, and hence $\alpha$ also vanishes. It follows from this and the decomposition in~\Cref{eq:partial-projection} that $\operatorname{im} G_3 = K\oplus\kk( u_1 \otimes u_2)$. The dimension statement is straightforward.

Finally, since the span of the set $\{\pi_1(h_i)\otimes\pi_2(h_i):i\in S\}$ contains $\operatorname{im}G_3$, and $|S|\leq m+1= \dim \operatorname{im}G_3$, it follows that $|S|=m+1$ and this is a basis for $\operatorname{im}G_3=K\oplus\kk( u_1 \otimes u_2)$. This completes the proof.
\end{proof}

\begin{lemma}\label{lem:binary}
Let \(W=\mathbb{F}^2\), where
\(\mathbb{F}\in\{\mathbb{R},\CC\}\).
The following statements hold:
\begin{enumerate}
\item If \(Q\in S^2W\) has rank two, then for every
\(F\in S^3W\) there exists \(h\in W\) such that
\[
F-h^3\in WQ.
\]

\item Let \(x,y\) be a basis of \(W\). If
\[
F=\alpha y^3+\beta xy^2+\gamma x^2y+\delta x^3
\]
with \(\alpha\ne0\), then there exists \(h\in W\) such that
\[
F-h^3\in Wx^2.
\]
\end{enumerate}
\end{lemma}

\begin{proof}
For (i), first suppose that we can write \(Q=uv\) for some $u,v \in W$. This is always the case when \(\mathbb{F}=\CC\). Since $Q$ has rank two, the vectors \(u,v\) are linearly independent and hence form a basis of \(W\). Thus we can write
\[
WQ=\operatorname{span}\{u^2v,uv^2\}.
\]
Write an arbitrary cubic as
\[
F=\alpha u^3+\beta u^2v+\gamma uv^2+\delta v^3.
\]
Let \(a,b\in\mathbb{F}\) be cube roots of \(\alpha,\delta\), respectively.
Then
\[
(au+bv)^3
=
\alpha u^3+3a^2b\,u^2v+3ab^2\,uv^2+\delta v^3.
\]
Consequently,
\[
F-(au+bv)^3
=
(\beta-3a^2b)u^2v+(\gamma-3ab^2)uv^2
\in WQ.
\]

Now suppose that \(\mathbb{F}=\mathbb{R}\) and we cannot write \(Q=uv\) for $u,v \in W$. Replacing \(Q\) by \(-Q\) if necessary,
which does not change \(WQ\), we can factor \(Q=uv\) in the
complexification of \(W\), with \(v=\overline{u}\).
Write \(F\) in this basis as above. Since \(F\) is real,
\[
\delta=\overline{\alpha},\qquad \gamma=\overline{\beta}.
\]
Choose \(a\in\CC\) with \(a^3=\alpha\). Then \(h:=au+\overline{a} v\) belongs to \(W\), and
the preceding calculation gives
\[
F-h^3
=
\bigl((\beta-3a^2 \overline{a})u+(\gamma-3a \overline{a}^2)v\bigr)Q
\in WQ,
\]
because the vector in parentheses is fixed by complex conjugation.
This proves (i).

For (ii), note that
\[
Wx^2=\operatorname{span}\{x^3,x^2y\}.
\]
Choose \(b\in\mathbb{F}\) with \(b^3=\alpha\), and let
\(a=\frac{\beta}{3b^2}\).
Expanding gives
\[
(by+ax)^3
=
b^3y^3+3b^2a\,xy^2+3ba^2x^2y+a^3x^3
=
\alpha y^3+\beta xy^2+3ba^2x^2y+a^3x^3,
\]
and hence,
\[
F-(by+ax)^3
=
(\gamma-3ba^2)x^2y+(\delta-a^3)x^3
\in Wx^2.
\]
This proves (ii).
\end{proof}

%
%
%
%

\begin{lemma}\label{lem:residual-plane}
Suppose that $\cA$ contains neither zero nor a symmetric rank one tensor, but it does contain a (non-symmetric) rank one tensor
$P:= u_1\otimes u_2\otimes u_3\in\cA$. Let
\[
 W:=\spn\{u_1, u_2, u_3\},
 \qquad  K_W:= K\cap\Sym^2 W.
\]
Then $\dim W=2$ and
\begin{equation*}
  K_W=\{0\}
 \qquad \qquad \text{or}\qquad \qquad
  K_W=\kk x^{2}\quad
 \text{for some }\quad 0\ne x\in W.
\end{equation*}
In the second case, exactly one factor of $P$ is proportional to $ x$.
\end{lemma}
\begin{proof}
Since $T_0 \in S^3 V$ is symmetric, the map $V^{\otimes3}\to\Lambda^3V$ annihilates $T_0$, and similarly it annihilates each space $\iota_i(V \otimes K)$, since $K$ is spanned by symmetric matrices. Thus
$ u_1\wedge u_2\wedge u_3=0$, so the vectors $u_1,u_2,u_3$ are linearly dependent. They are not all proportional, since that
would make $P$ a symmetric rank-one tensor. Hence $W$ is two-dimensional.

Since $T_0$ is symmetric and $\sym_3(J(K)) \subseteq J(K)$, it follows that $\sym_3(\cA) \subseteq \cA$. Hence $\sym_3(P)\in\cA$. Note that $\sym_3(P) \in S^3 W.$ If $K_W$ contained a rank-two matrix $Q$, then \Cref{lem:binary}(i) would give some $h\in W$ such that
\[
    \sym_3(P)-h^3\in WQ\subseteq VK\subseteq J(K).
\]
Since $\sym_3(P)\in\cA=T_0+J(K)$, it follows that $h^3\in\cA$. Thus either
$h=0$, giving $0\in\cA$, or $h\neq0$, giving a nonzero symmetric rank one
tensor in $\cA$, both of which are forbidden. So $K_W$ does not contain a rank-two matrix.

The only way for $K_W$ to not contain a rank-two matrix is $K_W=\{0\}$ or $K_W=\kk x^{2}$ for some nonzero vector $x\in W$. It remains to prove that exactly one factor of $P$ is proportional to $x$ in the latter case. If two factors of $P$ were
proportional to $ x$, then $P\in J(K)$. For example, if $P=x \otimes x \otimes u$, then $P \in \iota_3(V \otimes x\otimes x) \subseteq \iota_3(V \otimes K)$. Since $P$ also lies in
$T_0+ J(K)$, this would imply $0\in\cA$, a contradiction. If no factors of $P$ are proportional
to $ x$, complete $ x$ to a basis $ x, y$ of $W$.
Then $\Pi_3(P)$ has a
nonzero $ y^{3}$ coefficient. By \Cref{lem:binary}(ii), $\Pi_3(P)$ is equal to a cube modulo $W\, x^{2}\subseteq V\, K \subseteq J(K)$. Similarly to before, this implies that $\cA$ either contains $0$ or a symmetric rank-one
tensor, both of which are forbidden. Hence, exactly one factor of $P$ is proportional to $ x$. This completes the proof.
\end{proof}

\begin{proof}[Proof of~\Cref{thm:threshold}]
The statements about $\rk_{\kk}(T)$ follow from~\cite[Lemma 5]{Shitov2018}, or alternatively \Cref{lem:projections}. It remains to prove the statements about $\srkfield{\kk}(T)$.

First, suppose that $\cA$ contains a symmetric rank two tensor $S$. Since $S-T_0 \in J(K) \cap S^3 V = VK$, there exist vectors $b_1,\dots, b_m \in V$ such that
\[
 S-T_0=3 \sum_{\nu=1}^m  b_\nu\, w_\nu^2.
\]
Recall that we also have
\[
T=T_0+3 \sum_{\nu=1}^m \lambda_{\nu} e_{\nu} w_{\nu}^2
\]
for some scalars $\lambda_{\nu} \in \kk$.  Consequently,
\[
 T=S+3\sum_{\nu=1}^m
          ( \lambda_{\nu} e_\nu- b_\nu)\, w_\nu^{2}.
\]
Using the identity
\[
3zw^2
=
\frac12(w+z)^3-\frac12(w-z)^3-z^3
\]
we obtain $\srkfield{\kk}(T) \leq 3m+2$.

It remains to prove that $\srkfield{\kk}(T)\ge3m+2$ under the sole assumption that $\cA$ contains neither zero nor a symmetric rank one tensor. We note that the below proof holds even if the $M_{\nu}$ do not have rank one. Suppose toward contradiction that
\begin{equation*}
 T=\sum_{\ell=1}^N h_\ell^{3},
 \qquad N\le3m+1,
\end{equation*}
with every $h_{\ell}\neq 0$. By~\Cref{lem:partial-projection}, there exists a rank-one tensor $P=u_1 \otimes u_2 \otimes u_3 \in \cA$. Let
\[
 W:=\spn\{u_1, u_2, u_3\},
 \qquad  K_W:= K\cap\Sym^2 W
\]
as in~\Cref{lem:residual-plane}. Then
\begin{equation*}
  K_W=\{0\}
 \qquad \qquad \text{or}\qquad \qquad
  K_W=\kk x^{2}\quad
 \text{for some }\quad 0\ne x\in W.
\end{equation*}

\emph{Case 1: $K_W=\{0\}$.} Since $\cA$ does not contain any symmetric rank one tensor, there must exist two vectors among $u_1,u_2,u_3$ that are linearly independent. Permuting factors if necessary, we may assume that these are $u_1$ and $u_2$.  Since $K \subseteq S^2V$ is symmetric, we have $u_1 \otimes u_2 \notin K$. Let
\[
  a\otimes b=\beta u_1\otimes u_2+D,
 \qquad \beta\ne0,\quad D\in K
\]
be a rank-one matrix in $K \oplus \kk (u_1\otimes u_2)$ outside $K$. Since $D$ is symmetric, antisymmetrization gives
$ a\wedge b=\beta u_1\wedge u_2\ne0$. Hence, $ a, b\in \spn\{u_1,u_2\}$ and
$D\in K\cap\Sym^2 W=0$, so $a \otimes b$ is a scalar multiple of $u_1 \otimes u_2$. It follows that all but exactly one element of the rank-one basis in~\Cref{eq:rank-one-basis} from \Cref{lem:partial-projection} lies in $K$, and the remaining element is a scalar multiple of $u_1 \otimes u_2$. Choose $\ell \in S$ so that $\pi_1(h_\ell)\otimes\pi_2(h_\ell) = \beta u_1 \otimes u_2$ for some $\beta \neq 0$. Then
\[
    G:=(\pi_1 \otimes \pi_2 \otimes \id_{\widetilde{V}}) (T) =\beta(u_1\otimes u_2)\otimes h_\ell
      +\sum_{i\in S\setminus\{\ell\}}
       \pi_1(h_i)\otimes\pi_2(h_i)\otimes h_i,
\]
where the matrices $\pi_1(h_i)\otimes\pi_2(h_i) \in K$ for all $i \neq \ell$. Comparing this with \Cref{eq:partial-projection} and recalling that $M_1,\dots, M_m$ also forms a basis for $K$, we obtain
\begin{equation*}
    u_3=\beta h_\ell.
\end{equation*}
Since $u_3\in V$, we have $h_\ell\in V$.
Consequently,
\[
    h_\ell\otimes h_\ell
    =\pi_1(h_\ell)\otimes\pi_2(h_\ell)
    =\beta u_1\otimes u_2.
\]
This is impossible: the left-hand side is symmetric, whereas the
right-hand side is not.

\emph{Case 2: $K_W=\kk x^{2}$.}
By \Cref{lem:residual-plane}, exactly one factor of $P$ is proportional
to $x$. Permuting factors if necessary, we may assume
\[
    P=x\otimes u_2\otimes u_3,
    \qquad u_2,u_3\notin\kk x.
\]
In particular, $x,u_2$ are linearly independent and span $W$. Let
\[
    a\otimes b=\beta x\otimes u_2+D,
    \qquad \beta\ne0,\quad D\in K
\]
be a rank-one matrix in $K \oplus \kk (x \otimes u_2)$ outside $K$. Since $D$ is symmetric, antisymmetrization gives
\[
    a\wedge b=\beta x\wedge u_2\ne0.
\]
Hence $a,b\in W$, and
\[
    D\in K\cap\Sym^2W=\kk x^{2}.
\]
Consequently, for some $\delta\in\kk$,
\[
    a\otimes b
    =\beta x\otimes u_2+\delta x^{2}
    =x\otimes(\beta u_2+\delta x).
\]

Let
\[
    B:=\{i\in S:
       \pi_1(h_i)\otimes\pi_2(h_i)\notin K\}
\]
index the elements of the rank-one basis in~\Cref{eq:rank-one-basis} that lie outside $K$.
For each $i\in B$, the preceding argument gives scalars
$\beta_i\ne0$ and $\delta_i$ such that
\begin{equation}\label{eq:outside-double-line}
    \pi_1(h_i)\otimes\pi_2(h_i)
    =\beta_i x\otimes u_2+\delta_i x^{2}
    =x\otimes(\beta_i u_2+\delta_i x).
\end{equation}
Thus the surviving decomposition can be written as
\[
\begin{aligned}
    G={}&(x\otimes u_2)\otimes
             \left(\sum_{i\in B}\beta_i h_i\right)
       +x^{2}\otimes
             \left(\sum_{i\in B}\delta_i h_i\right)\\
       &+\sum_{i\in S\setminus B}
          \bigl(\pi_1(h_i)\otimes\pi_2(h_i)\bigr)\otimes h_i.
\end{aligned}
\]
Since $x^{2}\in K$ and every grouped matrix in the last sum
belongs to $K$, comparing this with \Cref{eq:partial-projection} gives
\[
    u_3=\sum_{i\in B}\beta_i h_i.
\]

For each $i\in B$, \Cref{eq:outside-double-line} is an equality
of nonzero rank-one matrices whose right-hand side has first factor
$x$. Therefore
\[
    \pi_1(h_i)\in\kk x.
\]
Since $u_3\in V$ and $\pi_1$ fixes $V$, applying $\pi_1$ to the
preceding expression for $u_3$ yields
\[
    u_3=\pi_1(u_3)
       =\sum_{i\in B}\beta_i\pi_1(h_i)
       \in\kk x.
\]
This contradicts $u_3\notin\kk x$ and proves the symmetric lower
bound in both cases.
\end{proof}

\section{The explicit construction}\label{sec:grid}

To find a tensor whose symmetric rank is not equal to its rank, it suffices by~\Cref{thm:threshold} to construct a choice of data in~\Cref{def:adjoining} such that $\cA$ contains neither zero nor a symmetric rank-one tensor, but does contain a (non-symmetric) rank one tensor. For our choice of parameters, $\cA$ will contain a symmetric rank two tensor, so the gap between the rank and the symmetric rank is one.

At a high level, our construction proceeds as follows. We first choose a certain collection of three linearly independent vectors
\begin{align*}
a_1,a_2,c\in V=\CC^{9}
\end{align*}
with rational entries, and let $K$ be the span of a certain set of eighteen linearly independent symmetric rank one matrices
\begin{align*}
M_\nu=\lambda_\nu w_\nu^{2}\in S^2V,
\end{align*}
again with rational entries, for which $\cA$ contains neither zero nor any symmetric rank one tensor, and
\begin{align}\label{eq:grid-sums-summary1}
    \sum_{\nu=1}^9 M_\nu
    =a_1^{2}-c^{2}, \qquad
    \sum_{\nu=10}^{18} M_\nu
    =a_2^{2}-c^{2}.
\end{align}
We then set
\[
T_0=a_1^{3}+a_2^{3}+c^{3}.
\]
Note that $\cA$ contains a rank-one tensor: Letting \(s=a_1+a_2+c\), notice that
\begin{align}
T_0-a_1 \otimes \sum_{\nu=1}^{9} M_\nu- a_2 \otimes \sum_{\nu=10}^{18} M_{\nu}
&=
T_0-a_1\otimes(a_1^{2}-c^{2})
   -a_2\otimes(a_2^{2}-c^{2}) \notag\\
&=s\otimes c\otimes c.
\label{eq:ordinary-completion-summary}
\end{align}

Note that $\cA$ also contains a symmetric rank two tensor. Indeed, we have
\[
(a_1+a_2)^3
=4 a_1^3+4a_2^3 - 3(a_1-a_2)(a_1^2-a_2^2).
\]
This observation combined with the inclusion
\[
a_1^2-a_2^2
=
(a_1^2-c^2)-(a_2^2-c^2)\in K
\]
shows that 
\begin{equation}\label{eq:symmetric-completion-summary}
    S:=\frac14(a_1+a_2)^3+c^3 = T_0 - \frac{3}{4}(a_1-a_2)(a_1^2-a_2^2) \in\cA.
\end{equation}

Finally, we let $e_1,\dots, e_{18}$ be a rational basis for $\CC^{18}$, and define
\[
T
=
T_0+\sum_{\nu=1}^{18}\sum_{j=1}^3\iota_j(e_\nu \otimes M_\nu) \in S^3(\CC^{9} \oplus \CC^{18}).
\]
Since \(m=\dim K=18\), \Cref{thm:threshold} gives
\[
\rk_{\CC}(T)=3m+1=55,
\qquad
\srkfield{\CC}(T)=3m+2=56.
\]
Furthermore, the matching decompositions are rational, so the same rank equalities hold
over \(\QQ\) and \(\RR\). Indeed, let $i_{\nu} = 1$ if $1\leq \nu \leq 9$ and let $i_{\nu}=2$ if $10 \leq \nu \leq 18$.
\Cref{eq:ordinary-completion-summary} gives
\[
T
=
s\otimes c \otimes c+\sum_{\nu = 1}^{18}\Bigl[
\iota_1((e_\nu+a_{i_\nu}) \otimes M_\nu)
+\iota_2(e_\nu \otimes M_\nu)
+\iota_3(e_\nu \otimes M_\nu)
\Bigr].
\]
Since each $M_\nu$ is a rational rank one tensor, this is a rational decomposition of
length $1+3\cdot18=55.$

To find a rational symmetric decomposition of length $56$, let
\[
z_\nu
:=
e_\nu+\frac{(-1)^{i_{\nu}-1}}{4}(a_1-a_2),
\]
and recall from~\Cref{eq:grid-sums-summary1,eq:symmetric-completion-summary} that
\[
T_0-S=\frac{3}{4}(a_1-a_2)(a_1^2-a_2^2)=\frac{3}{4} \sum_{\nu=1}^{18} (-1)^{i_{\nu}-1} \lambda_{\nu} (a_1-a_2) w_{\nu}^2.
\]

Then
\begin{align*}
T-S&=(T_0-S)+3\sum_{\nu=1}^{18}\lambda_\nu e_\nu w_\nu^2\\
&=\frac34\sum_{\nu=1}^{18}
(-1)^{i_{\nu}-1} \lambda_\nu(a_1-a_2)w_\nu^2
+3\sum_{\nu=1}^{18}\lambda_\nu e_\nu w_\nu^2\\
&=
3\sum_{\nu=1}^{18}\lambda_\nu
\left(
e_\nu+\frac{(-1)^{i_{\nu}-1}}4(a_1-a_2)
\right)w_\nu^2\\
&=3\sum_{\nu=1}^{18}\lambda_\nu z_\nu w_\nu^2.
\end{align*}
By~\Cref{eq:symmetric-completion-summary}, the tensor $S$ is a sum of two rationally weighted symmetric rank one tensors, and each
term $3z_\nu w_\nu^2$ is a sum of three rationally weighted symmetric rank one tensors by
\[
3zw^2
=
\frac12(w+z)^3-\frac12(w-z)^3-z^3.
\]
Thus $T$ has a rational symmetric decomposition of length $2+3\cdot18=56$.

\subsection{The construction in detail}

From the discussion above, it suffices to construct rational $a_1, a_2, c \in V$ and linearly independent symmetric rank one matrices $M_1,\dots, M_{18} \in S^2(V)$ for which
\begin{enumerate}
\item \Cref{eq:grid-sums-summary1} holds.\label{item:item-one}
\item $\cA$ contains neither zero nor any symmetric rank one tensor.
\end{enumerate}
We now detail our specific construction, and prove that it satisfies these properties.

Let $d_0,d_1,d_2$ be the standard basis of $H=\CC^3$, and let
\begin{equation*}
 w_{rs}=(d_r,-2d_s)\in H\oplus H,
 \qquad G=\spn\{w_{rs}:0\le r,s\le2\}.
\end{equation*}
Writing $\sigma(u)=u_0+u_1+u_2$, we have
\begin{equation}\label{eq:grid-hyperplane}
 G=\{(u,v):2\sigma(u)+\sigma(v)=0\},\qquad \dim G=5.
\end{equation}
Let
\[
 \alpha=d_0+d_1-2d_2,
 \qquad (\alpha_0,\alpha_1,\alpha_2)=(1,1,-2),
\]
and let
\begin{equation*}
 a:=(\alpha,-\alpha)\in G,\qquad c:=(\alpha,\alpha)\in G.
\end{equation*}
Let
\[
 K_G:=\spn\{w_{rs}^{2}:0\le r,s\le2\}\subseteq S^2G.
\]

Now let $E\subseteq G$ be the orthogonal complement to $c$ (or any complement to $\CC c$ with a rational basis), so that
$G=\CC c\oplus E$. Let $E_1,E_2$ be two copies of $E$. Define
\begin{equation*}
 V:=\CC c\oplus E_1\oplus E_2.
\end{equation*}
For $i=1,2$, identify $G$ with the subspace
$G_i:=\CC c\oplus E_i\subseteq V$ by sending $tc+e$ to $tc+e_i$,
where $e_i$ denotes the copy of $e\in E$ in $E_i$. Thus
\[
 G_1\cap G_2=\CC c,\qquad \dim V=9.
\]
Let $a_i,w_{i,r,s}$ be the corresponding copies of $a,w_{r,s}$ in
$G_i$, and put
\begin{equation*}
 T_0:=a_1^{3}+a_2^{3}+c^{3},
 \qquad K:=K_1+K_2,
\end{equation*}
where $K_i\subseteq S^2G_i$ is the copy of $K_G$.
Note that the vectors $a_1,a_2,c$ are linearly independent.

For $\nu=(i,r,s)\in I:=\{1,2\}\times\{0,1,2\}^2$, let
\[
 i_\nu:=i,\qquad w_\nu:=w_{i,r,s},\qquad
 \lambda_\nu:=\alpha_r\alpha_s,\qquad
 M_\nu:=\lambda_\nu w_\nu^{2}.
\]
Let $Z=\CC^{I}$ with standard basis $(e_\nu)_{\nu\in I}$.


We first verify that the matrices $M_{\nu}$ indeed form a basis for $K$, and that~\Cref{item:item-one} holds.

\begin{lemma}\label{lem:grid-facts1}
The 18 matrices $w_{irs}^2$ form a basis for $K$, and
\begin{equation*}
    \sum_{r,s=0}^2 M_{(i,r,s)}
    =a_i^{2}-c^{2} \qquad \text{for}\qquad i=1,2.
\end{equation*}
\end{lemma}
\begin{proof}
Note that
\[
w_{rs}^{2}
=
\begin{pmatrix}
d_r d_r^{\mathsf T}
&
-2 d_r d_s^{\mathsf T}
\\[2pt]
-2 d_s d_r^{\mathsf T}
&
4 d_s d_s^{\mathsf T}
\end{pmatrix}.
\]
The upper-right blocks are linearly independent, hence the matrices $w_{rs}^{2}$ are linearly independent. Note also that every matrix in $K_G$ has diagonal upper-left and lower-right blocks,
whereas both of these blocks of $c^2$ are
$\alpha\alpha^{\mathsf T}$, which is not diagonal. Hence
$c^2\notin K_G$.

Since $G_i=\CC c\oplus E_i$, we have
\[
    S^2G_1\cap S^2G_2=\CC c^2,
\]
so $K_1\cap K_2\subseteq \CC c^2$. Since $c^2$ is not contained in $K_1$ nor $K_2$, we have $K_1\cap K_2=\{0\}$. Hence, the union of the two bases $\{w_{irs}^2\}_{r,s} \subseteq K_i$ forms a basis for $K$.

Finally, since $\sum_r\alpha_r=0$, we obtain
\begin{align}\label{eq:moment}
    \sum_{r,s=0}^2\alpha_r\alpha_s w_{rs}^2
    =
    \begin{pmatrix}
        0 & -2\alpha\alpha^{\mathsf T}\\
        -2\alpha\alpha^{\mathsf T} & 0
    \end{pmatrix}
    =a^2-c^2,
\end{align}
which implies
\[
    \sum_{r,s=0}^2 M_{(i,r,s)}=a_i^2-c^2,
    \qquad i=1,2,
\]
as required.
\end{proof}

It remains only to prove the following.
\begin{theorem}\label{thm:no-cube}
Let $T_0 \in S^3V$ and $K \subseteq S^2V$ be as above. Then the affine space $\cA=T_0+J(K) \subseteq V^{\otimes 3}$ contains neither zero nor a symmetric
rank-one tensor.
\end{theorem}

We now give an overview of the proof of this theorem. The goal is to prove that $h^3 \notin \cA$ for any vector $h \in V$. We assume toward contradiction that $h^3 \in \cA$. We first establish that if $h\notin G_1\cup G_2$ then one can write $h=h_1+h_2$, where each $h_i$ is a nonzero multiple of $a_i$ or some $w_{irs}$~(\Cref{lem:mixed}). We then construct three linear maps $R_1,R_2,R_3:V\to H$ for which $(R_j^{\otimes3}P)_{0,1,2}=-2$ for every $P\in\cA$ and $j=1,2,3$, immediately showing that $0\notin \cA$. Hence, if $h^3 \in \cA$, then $(R_jh)_0(R_jh)_1(R_jh)_2=-2$ for all $j=1,2,3$. These equations rule out $h\in G_1\cup G_2$ and each of the three remaining cases: both $h_i$ are multiples of the $a_i$, exactly one is, or both are multiples of the $w_{irs}$.

\subsubsection{Lemmas for the proof of~\Cref{thm:no-cube}}

We require two lemmas to prove~\Cref{thm:no-cube}.

\begin{lemma}\label{lem:grid-facts}
It holds that $K_G\cap cG=0$, and for $h \in G$ it holds that
\begin{equation*}
 h^2\in K_G+\CC c^2
 \quad\Longleftrightarrow\quad
 h\in\CC a\cup\CC c\cup\bigcup_{r,s}\CC w_{rs}.
\end{equation*}
\end{lemma}
\begin{proof}
Let $h=(u,v)\in G$ be arbitrary. For the first statement, it suffices to prove that $(ch \in K_G \Rightarrow h=0)$. Recall that every matrix in $K_G$ has diagonal upper-left and lower-right
blocks. The upper-left block of $ch$ is given by
\[
\frac12\bigl(\alpha u^{\mathsf T}+u\alpha^{\mathsf T}\bigr).
\]
If $ch \in K_G$, then the off-diagonal elements of this block are zero, meaning $\alpha_j u_k + u_j \alpha_k=0$ for all $j \neq k$. Since $\alpha_j \neq 0$ for all $j$, this implies $u=0$. An analogous argument for the lower-right blocks shows $v=0$. This proves the first statement.

For the second statement, if $h^2-\mu c^2\in K_G$, then
%
%
\[
 u_ju_k=\mu\alpha_j\alpha_k,
 \qquad v_jv_k=\mu\alpha_j\alpha_k
\]
for all $j \neq k$. When $\mu\ne0$, this implies that the ratios $u_j/\alpha_j$ are all equal to some $t\in \CC$. Similarly, $v=s\alpha$, with
$t^2=s^2=\mu$. Thus $h=tc$ or $h=ta$. When $\mu=0$, $u$ and $v$ have at most one nonzero coordinate each.
Write $u= \tau d_r$, $v= \gamma d_s$. \Cref{eq:grid-hyperplane} gives $\gamma =-2 \tau$, so $h=\tau w_{rs}$.
The converse follows from \Cref{eq:moment} and the definition of $K_G$.
\end{proof}

\begin{lemma}\label{lem:mixed}
If $h \in V$ satisfies
\[
 h\notin G_1\cup G_2,\qquad h^3\in S^3 G_1 + S^3 G_2 + E_2 K_1+ E_1 K_2,
\]
then we can write $h=h_1+h_2$, where
\[
h_i \in \CC a_i \cup\bigcup_{r,s}\CC w_{irs}\qquad\text{for}\qquad i=1,2.
\]

\end{lemma}

\begin{proof}
Recall that
\[
    G_i=\CC c\oplus E_i,
    \qquad
    V=\CC c\oplus E_1\oplus E_2.
\]
We may therefore write uniquely
\[
    h=\tau c+x+y,
    \qquad x\in E_1,\quad y\in E_2.
\]
Since $h\notin G_1\cup G_2$, both $x$ and $y$ are nonzero.
By assumption, we can write
\begin{align}\label{eq:5terms}
    h^3=C+B_{21}+B_{12},
\end{align}
where
\[
C \in S^3 G_1 + S^3 G_2, \qquad
    B_{21}\in E_2K_1,\qquad
    B_{12}\in E_1K_2.
\]

\begin{claim}\label{claim:construction}
It holds that $B_{21}=3y Q_1$ and $B_{12}=3x Q_2$, where
\begin{equation*}
    Q_1=x^2+2cb_1+\mu_1c^2 \in K_1
\end{equation*}
for some $b_1\in E_1$ and $\mu_1\in\CC$, and
\begin{equation*}
    Q_2=y^2+2cb_2+\mu_2c^2 \in K_2
\end{equation*}
for some $b_2\in E_2$ and $\mu_2\in\CC$.
\end{claim}

Before proving the claim, we use it to complete the proof of the lemma. Note that we have the direct sum decomposition
\[
\begin{aligned}
S^3V
={}\CC c^3
\oplus c^2E_1
\oplus c^2E_2
&\oplus cS^2E_1
\oplus cE_1E_2
\oplus cS^2E_2\\
&\oplus S^3E_1
\oplus S^2E_1E_2
\oplus E_1S^2E_2
\oplus S^3E_2.
\end{aligned}
\]
The $cE_1E_2$ component of
\[
    h^3=(\tau c+x+y)^3
\]
is given by $6\tau cxy.$ Since the space $S^3 G_1+S^3G_2$ has no $cE_1 E_2$ component, it follows that this component of $h^3$ is also equal to the $cE_1 E_2$ component of $3yQ_1+3xQ_2$, which is given by
\[
    6c(b_1y+xb_2).
\]
Hence,
\begin{equation}\label{eq:mixed-linear}
    b_1y+xb_2=\tau xy.
\end{equation}

Let $\varphi\in (E_1\oplus E_2)^*$ satisfy $\varphi(E_2)=0$ and $\varphi(x)=0$.
Applying $\varphi\otimes\id$ to
\Cref{eq:mixed-linear} gives
\[
    \varphi(b_1)y=0.
\]
Since $y\ne0$, every such $\varphi$ annihilates $b_1$, and therefore
$b_1\in\CC x$. Similarly, $b_2\in\CC y$. Write
\[
    b_1=tx,\qquad b_2=sy
\]
so that $t+s=\tau.$

Define
\[
    h_1=x+tc\in G_1,
    \qquad
    h_2=y+sc\in G_2.
\]
Then
\[
    h_1+h_2=x+y+(t+s)c=h.
\]
Moreover,
\[
\begin{aligned}
    h_1^2-Q_1
      &=(t^2-\mu_1)c^2,\\
    h_2^2-Q_2
      &=(s^2-\mu_2)c^2.
\end{aligned}
\]
Since $Q_i\in K_i$, it follows that
\[
    h_i^2\in K_i+\CC c^2
    \qquad\text{for}\qquad i=1,2.
\]
Finally, $h_1$ has the nonzero $E_1$ component $x$, and $h_2$ has the
nonzero $E_2$ component $y$, so neither vector belongs to $\CC c$. By~\Cref{lem:grid-facts}, each $h_i$ must be of the specified form. It remains only to prove the claim.

\begin{proof}[Proof of~\Cref{claim:construction}]
For $i=1,2$ let
\[
    \rho_i:S^2G_i\longrightarrow S^2E_i
\]
be induced by the projection $G_i\to E_i$ with kernel $\CC c$. Note that
\[
S^2 G_i=\CC c^2 \oplus c E_i \oplus S^2E_i= cG_i \oplus S^2E_i.
\]
Hence,
\[
    \ker\rho_i=cG_i.
\]
By~\Cref{lem:grid-facts} we have $K_i\cap cG_i=0$, hence the restriction
\[
    \rho_i|_{K_i}:K_i\longrightarrow S^2E_i
\]
is injective.

%
%
We first compute $B_{21}$. Note that the $S^2E_1E_2$ component of $h^3$ is $3x^2y$, and among the three terms in~\Cref{eq:5terms}, only $B_{21}$ can contribute to $S^2E_1E_2$. Let
\[
    B_{21}=\sum_j y_jQ_j,
    \qquad y_j\in E_2,\quad Q_j\in K_1.
\]
Then the $S^2E_1E_2$ component of $B_{21}$ is
\[
    \sum_j y_j\rho_1(Q_j).
\]
Since this component must equal $3x^2y$, and
$\rho_1|_{K_1}$ is injective, it follows that
\[
    B_{21}=3yQ_1
\]
for a unique $Q_1\in K_1$ satisfying $\rho_1(Q_1)=x^2$. Thus
\begin{equation*}
    Q_1=x^2+2cb_1+\mu_1c^2
\end{equation*}
for some $b_1\in E_1$ and $\mu_1\in\CC$. The statement for $B_{12}$ is proven similarly.
%
\end{proof}
The proof of the claim completes the proof of the lemma.
\end{proof}

\subsubsection{Proof of~\Cref{thm:no-cube}}

In this section we prove~\Cref{thm:no-cube}.

\begin{proof}[Proof of~\Cref{thm:no-cube}]

Let $\operatorname{pr}_1,\operatorname{pr}_2:G\to H$ be the coordinate projections, and let  $R_1,R_2,R_3:V\to H$ be defined as follows:
\begin{equation*}
 \begin{array}{c|cc}
 &G_1&G_2\\\hline
 R_1&\operatorname{pr}_2&\operatorname{pr}_1\\
 R_2&\operatorname{pr}_1&\operatorname{pr}_2\\
 R_3&-\operatorname{pr}_2&-\operatorname{pr}_2.
 \end{array}
\end{equation*}
The restrictions agree on $c$, so these are well-defined linear maps.
Their values on the vectors $a_1, a_2, c$ are
\begin{equation*}
 \begin{array}{c|ccc}
 &a_1&a_2&c\\\hline
 R_1&-\alpha&\alpha&\alpha\\
 R_2&\alpha&-\alpha&\alpha\\
 R_3&\alpha&\alpha&-\alpha.
 \end{array}
\end{equation*}

Let $P \in \cA$ be arbitrary. Note that $R_j^{\otimes3}(T_0)=\alpha^{3}$ for every $j$. In particular, $R_j^{\otimes3}(T_0)_{0,1,2}=\alpha_0\alpha_1\alpha_2=-2.$
On the other hand, the coordinate projections $\operatorname{pr}_1,\operatorname{pr}_2$ send each
$w_{rs}$ to a multiple of a basis vector, hence  $R_j^{\otimes 2}(K) \subseteq \spn\{d_0^2, d_1^2, d_2^2\}.$ It follows that $R_j^{\otimes3}(D)_{0,1,2}=0$ for all $D \in J(K)$, so $R_j^{\otimes 3}(P)_{0,1,2}=-2$ for every $P \in \cA$. Hence, $0 \notin \cA$.

%

Suppose toward contradiction that there exists a nonzero symmetric rank one tensor $h^{3}\in \cA$. Then
\begin{align}\label{eq:phi-replacement}
R_j^{\otimes 3}(h^{3})_{0,1,2}=-2\quad \text{for} \quad j=1,2,3.
\end{align}
Since $R_3=-R_1$ on $G_1$, and $R_3=-R_2$ on $G_2$, it follows that $h \notin G_1 \cup G_2$.

The difference $h^3-T_0$ is symmetric, so it belongs to
$J(K)\cap S^3V=VK_1+VK_2$. Since $T_0 \in S^3 G_1 + S^3 G_2$, and
\[
VK_1 = G_1 K_1 \oplus E_2 K_1 \subseteq S^3 G_1 + E_2 K_1
\]
(and similarly for $VK_2$), we obtain
\[
h^3 \in S^3G_1 + S^3 G_2 + E_2 K_1+E_1 K_2.
\]
By \Cref{lem:mixed}, we can write $h=h_1+h_2$
where each $h_i$ is a nonzero multiple of either $a_i$ or some $w_{irs}$.
We break this into three cases, obtaining a contradiction in each one.

\emph{Case 1:} If $h=b_1a_1+b_2a_2$ for some $b_1,b_2 \in \CC$, then $R_1h=-R_2h$, contradicting
\Cref{eq:phi-replacement}.

\emph{Case 2:} If $h=ba_2+\gamma w_{1,r,s}$ with $b\gamma\ne0$,
then $R_1h=b\alpha-2\gamma d_s$ and
$R_3h=b\alpha+2\gamma d_s$. Hence
\[
R_3^{\otimes 3}(h^{3})_{0,1,2}-R_1^{\otimes 3}(h^{3})_{0,1,2}
   =4\gamma b^2\prod_{k\ne s}\alpha_k\ne0,
\]
contradicting \Cref{eq:phi-replacement}. The case $h=ba_1+\gamma w_{2,r,s}$ with $b\gamma\ne0$ is proven similarly with $R_1$ replaced by $R_2$.

\emph{Case 3:} If $h=\gamma_1w_{1,r,s}+\gamma_2w_{2,r',s'}$, then
$R_3h=2\gamma_1d_s+2\gamma_2d_{s'}$ has at most two nonzero
coordinates. Thus $R_3^{\otimes 3}(h^{3})_{0,1,2}=0$, contradicting \Cref{eq:phi-replacement}.
\end{proof}

\section*{Appendix}
\appendix

\section{Proof of~\Cref{lem:projections}}\label{sec:substitution}

In this section we prove~\Cref{lem:projections}. We first restate the lemma for convenience.

\begin{lemma}[\Cref{lem:projections}]
Let
\begin{align}\label{eq:decomp}
T= \sum_{i=1}^N a_i\otimes b_i\otimes c_i
\end{align}
be an expression of the symmetrically adjoined tensor $T$ as a sum of $N$ product tensors. Then there exist linear maps
\[
 \pi_j:\widetilde V\longrightarrow V,
 \qquad \pi_j|_V=\id_V,\qquad j=1,2,3,
\]
and pairwise disjoint subsets
\[
I_1, I_2, I_3 \subseteq [N], \qquad |I_j|=m
\]
such that the following statements hold:
\begin{enumerate}
\item $\pi_1 a_i=0$ for all $i \in I_1$, $\quad$ $\pi_2 b_i =0$ for all $i \in I_2$, \quad and \quad $\pi_3 c_i = 0$ for all $i \in I_3$.
\item $(\pi_1 \otimes \pi_2 \otimes \pi_3)(T) \in \cA$.
\item It holds that
\begin{equation*}
 \rk(T)=3m+\min\{\rk(P):P\in\cA\}.
\end{equation*}
\end{enumerate}
\end{lemma}
\begin{proof}
We construct the three maps
\(\pi_1,\pi_2,\pi_3:V\oplus Z\to V\) successively. Write
\[
a_i=a_i^V+a_i^Z,
\qquad
a_i^V\in V,\quad a_i^Z\in Z.
\]
We claim that the vectors \(a_1^Z,\ldots,a_N^Z\) span \(Z\).
Indeed, let \(\varphi\in Z^*\), extended by zero on \(V\). Contracting the
first factor of the decomposition in~\Cref{eq:decomp} gives
\[
(\varphi\otimes\id\otimes\id)(T)
=
\sum_{i=1}^N\varphi(a_i^Z)\,b_i\otimes c_i.
\]
On the other hand, from the definition of the adjoined tensor,
\[
(\varphi\otimes\id\otimes\id)(T)
=
\sum_{\nu=1}^m\varphi(e_\nu)M_\nu.
\]
Since \(M_1,\ldots,M_m\) are linearly independent, the latter expression
vanishes only when \(\varphi=0\). Hence no nonzero functional on \(Z\)
annihilates all the \(a_i^Z\), proving that they span \(Z\).

Choose indices \(I_1\subseteq\{1,\ldots,N\}\), with \(|I_1|=m\), such that
$
\{a_i^Z:i\in I_1\}
$
is a basis of \(Z\). There is a unique linear map \(B_1:Z\to V\) satisfying $B_1(a_i^Z)=a_i^V$ for all $i \in I_1$. Define
\[
\pi_1:V\oplus Z\to V,
\qquad
\pi_1(v+z)=v-B_1z.
\]
Then \(\pi_1|_V=\id_V\), and $\pi_1(a_i)=0$ for all $i \in I_1$.

We now repeat the same argument in the second factor. Since \(\pi_1\) fixes
\(V\), we have
\[
(\pi_1\otimes\id\otimes\id)(T)
=
T_0
+\sum_{\nu=1}^m \iota_1(\pi_1(e_\nu) \otimes M_\nu)
+\sum_{\nu=1}^m \iota_2(e_\nu \otimes M_\nu)
+\sum_{\nu=1}^m \iota_3( e_\nu \otimes M_\nu).
\]
Hence, if \(\varphi\in Z^*\) (extended by zero on $V$) is applied in the second factor, the result is
again
\[
\sum_{\nu=1}^m\varphi( e_\nu)M_\nu.
\]
The same argument as above shows that the \(Z\)-components $b_i^Z$, \(i\notin I_1\) span \(Z\).
We may therefore choose a set
\[
I_2\subseteq\{1,\ldots,N\}\setminus I_1,
\qquad |I_2|=m,
\]
and a map \(\pi_2:V\oplus Z\to V\), with
\(\pi_2|_V=\id_V\), such that $\pi_2(b_i)=0$ for all $i \in I_2$.

Finally, the same argument in the third factor shows that the \(Z\)-components
of the remaining \(c_i\)'s span \(Z\). Hence there is a third set
\[
I_3\subseteq\{1,\ldots,N\}\setminus(I_1\cup I_2),
\qquad |I_3|=m,
\]
and a map \(\pi_3:V\oplus Z\to V\), fixing \(V\), such that $\pi_3(c_i)=0$ for all $i \in I_3$. Consequently,
\[
(\pi_1\otimes\pi_2\otimes\pi_3)(T)
=
\sum_{i\notin I_1\cup I_2\cup I_3}
\pi_1(a_i)\otimes\pi_2(b_i)\otimes\pi_3(c_i),
\]
which has rank at most \(N-3m\).

Since every \(\pi_j\) fixes \(V\), applying the three maps to the defining
formula for \(T\) gives
\[
(\pi_1\otimes\pi_2\otimes\pi_3)(T)
=
T_0+\sum_{\nu=1}^m\sum_{j=1}^3
\iota_j(\pi_j( e_\nu) \otimes M_\nu).
\]
Each \(\pi_j( e_\nu)\) belongs to \(V\), so this tensor lies in $\cA=T_0+J(K)$. It follows that
\[
N-3m
\ge
\min\{R(P):P\in \cA\},
\]
so
\[
R(T)\ge
3m+\min\{R(P):P\in \cA\}.
\]

For the reverse inequality, choose
\(P\in \cA\) of minimum rank. Write
\[
P-T_0
=
\sum_{\nu=1}^m\sum_{j=1}^3
\iota_j(d_{\nu,j} \otimes M_\nu),
\qquad d_{\nu,j}\in V.
\]
Then
\[
T
=
P+\sum_{\nu=1}^m\sum_{j=1}^3
\iota_j( (e_\nu-d_{\nu,j}) \otimes M_\nu).
\]
Each \(M_\nu\) has rank one, so every tensor in the double sum is simple.
Therefore
\[
R(T)\le R(P)+3m.
\]
This completes the proof.
\end{proof}

\begingroup
\small
\setlength{\bibsep}{3pt}
\bibliographystyle{rankpaper}
\bibliography{references}
\endgroup
\end{document}